\documentclass[psreqno,reqno,11pt]{amsart}
\usepackage{eurosym}
\usepackage{amssymb}
\usepackage{tikz}
\usepackage{amsmath}
\usepackage{tikz}
\usepackage{pgflibraryplotmarks}
\usepackage{wasysym}
\usepackage{stmaryrd}
\usepackage[english]{babel}

\usepackage{hyperref}

\theoremstyle{plain}
\newtheorem{theorem}{Theorem}[section]
\newtheorem{proposition}[theorem]{Proposition}
\newtheorem{lemma}[theorem]{Lemma}
\newtheorem{corollary}[theorem]{Corollary}

\theoremstyle{definition}
\newtheorem{definition}[theorem]{Definition}
\newtheorem{notation}[theorem]{Notation}
\newtheorem{example}[theorem]{Example}

\allowdisplaybreaks

\begin{document}
	\def\N{\mathbb{N}}
	\def\Z{\mathbb{Z}}
	\def\Q{\mathbb{Q}}
	\def\R{\mathbb{R}}
	\def\P{\mathbb{P}}
	\def\E{E^{[2]}}
	\def\0{[0,\infty)}
	
	\title[Square closed pointed vector lattices]{
		Square closed pointed vector lattices}
	\author{C. Schwanke}
	\address{Department of Mathematics and Applied Mathematics, University of Pretoria, Private Bag X20, Hatfield 0028, South Africa}
	\email{cmschwanke26@gmail.com}
	\date{\today}
	
	\begin{abstract}
		Given an Archimedean vector lattice $E$, we present one elementary property of $E$ which is equivalent to the entire traditional list of axioms which makes $E$ a $\Phi$-algebra. We call a vector lattice with this property ``square closed". More generally, we then introduce the notion of a pseudo square closed vector lattice and prove that an Archimedean vector lattice is a semiprime $f$-algebra if and only if it is pseudo square closed. This theory serves as an efficient tool for determining whether or not an Archimedean vector lattice is a $\Phi$-algebra (or a semiprime $f$-algebra). To illustrate this point, we generalize a well-known result for uniformly complete Archimedean vector lattices with a strong order unit by proving that every functionally complete Archimedean vector lattice with a strong order unit is a $\Phi$-algebra. 
	\end{abstract}
	
	\maketitle
	\section{Introduction}\label{S:intro}
	
	A vector lattice $E$ is called a $\Phi$-algebra, or a unital $f$-algebra, with multiplicative unit $e\in E$, if $E$ is equipped with a map $E\times E\ni(f,g)\mapsto fg\in E$ that satisfies the following properties.
	\begin{itemize}
		\item[(i)] $f(g+h)=fg+fh\quad (f,g,h\in E)$.
		\item[(ii)] $(f+g)h=fh+gh\quad (f,g,h\in E)$.
		\item[(iii)] $(fg)h=f(gh)\quad (f,g,h\in E)$.
		\item[(iv)] $(\alpha f)(\beta g)=(\alpha\beta)fg\quad (f,g\in E, \alpha,\beta\in\R)$.
		\item[(v)] If $f,g\in E^+$, then $fg\in E^+$.
		\item[(vi)] If $f,g\in E$ satisfy $f\wedge g=0$, and $c\in E^+$, then $(cf)\wedge g=0$.
		\item[(vii)] If $f,g\in E$ are such that $f\wedge g=0$, and $c\in E^+$, then $(fc)\wedge g=0$.
		\item[(viii)] $fe=f\quad (f\in E)$.
		\item[(ix)] $ef=f\quad (f\in E)$.
	\end{itemize}
	
	For Archimedean $\Phi$-algebras, we replace this list of axioms with just one simple property in this paper, while simultaneously providing a purely order-theoretic description of Archimedean $\Phi$-algebras and their multiplication.
	
	The central idea for accomplishing this task is forming a completely order-theoretical depiction of the square function. In this prospect, one can observe that for any $x\in\R$, we have
	\begin{equation}\label{eq: x^2}
		x^2=\underset{\lambda\in\R}{\sup}\{2\lambda x-\lambda^2\}.
	\end{equation}
	
	The right hand side of formula \eqref{eq: x^2} also makes sense when $x\in\R$ is replaced with an element $f$ in vector lattice, and $\lambda^2$ is supplanted by $\lambda^2e$, for some designated element $e$ of the vector lattice. The caveat here is this vector lattice would need to be closed under this infinite supremum. We call such vector lattices \textit{square closed}, and we present Theorem~\ref{T: main}, which states that an Archimedean vector lattice is a $\Phi$-algebra if and only if it is square closed. Therefore, we illustrate that, for Archimedean vector lattices, the sole notion of square closedness is equivalent to the aforementioned nine properties of a $\Phi$-algebra.
	
	More generally, we also introduce the notion of a \textit{pseudo square closed} vector lattice. As a consequence of Theorem~\ref{T: main}, we prove that an Archimedean vector lattice is a semiprime $f$-algebra if and only if it is pseudo square closed in Theorem~\ref{T: semiprime f-alg}.
	
	A major benefit we gain from the theory presented in this paper is that one can determine whether or not an Archimedean vector lattice is a $\Phi$-algebra (semiprime $f$-algebra), simply by checking whether or not it is square closed (pseudo square closed). In this light, we show in Proposition~\ref{P: functionally complete} that a functionally complete Archimedean vector lattice with a strong order unit is a $\Phi$-algebra. Therefore, we generalize the well-known result that every uniformly complete Archimedean vector lattice with a strong order unit is a $\Phi$-algebra.
	
	We proceed with some preliminaries.
	
	\section{Preliminaries}\label{S: prelims}
	
	We refer the reader to the standard texts (e.g. \cite{AB, LuxZan1, LilZan, Zan2})  for any unexplained terminology or basic theory regarding vector lattices.
	Throughout this paper, $\mathbb{N}$ stands for the set of strictly positive integers, and the ordered field of real numbers is denoted by $\mathbb{R}$. All vector lattices in this document are real vector lattices.
	
	A vector space $E$ is called an \textit{associative algebra} if there exists a map
	\[
	E\times E\ni(f,g)\mapsto fg\in E,
	\]
	called a \textit{multiplication} on $E$, for which the following hold.
	\begin{itemize}
		\item[(i)] $f(g+h)=fg+fh\quad (f,g,h\in E)$.
		\item[(ii)] $(f+g)h=fh+gh\quad (f,g,h\in E)$.
		\item[(iii)] $(fg)h=f(gh)\quad (f,g,h\in E)$.
		\item[(iv)] $(\alpha f)(\beta g)=(\alpha\beta)fg\quad (f,g\in E, \alpha,\beta\in\R)$.
	\end{itemize}
	
	When an associative algebra $E$ is a vector lattice for which
	\begin{itemize}
		\item[(v)] 	$fg\in E^+$ for every $f,g\in E^+$,
	\end{itemize}
	we call $E$ an $\ell$-\textit{algebra}. An $f$-\textit{algebra} $E$ is an $\ell$-algebra which satisfies the following additional properties.
	\begin{itemize}
		\item[(vi)] If $f,g\in E$ satisfy $f\wedge g=0$, and $c\in E^+$, then $(cf)\wedge g=0$.
		\item[(vii)] If $f,g\in E$ are such that $f\wedge g=0$, and $c\in E^+$, then $(fc)\wedge g=0$.
	\end{itemize}
	If an $f$-algebra $E$ contains a multiplicative unit, that is, an element $e\in E$ such that
	\begin{itemize}
		\item[(viii)] $fe=f\quad (f\in E)$, and
		\item[(ix)] $ef=f\quad (f\in E)$,
	\end{itemize}
	then we call $E$ a $\Phi$-\textit{algebra} (or \textit{unital} $f$-\textit{algebra}).
	
	Clearly, the multiplicative unit in a $\Phi$-algebra has a certain distinction among the other elements. More generally, when a vector lattice has a particular element of interest, one can consider the notion of a pointed vector lattice.
	
	\begin{definition}\label{D: pos pointed AVL}
		We call a pair $(E,e)$, where $E$ is a vector lattice, and $e\in E$, a \textit{pointed vector lattice}. By a \textit{positively pointed} vector lattice, we mean a pointed vector lattice $(E,e)$ for which $e\in E^+$. If $E$ is Archimedean, then we refer to $(E,e)$ as a \textit{pointed Archimedean vector lattice}. Furthermore, if $E$ is Archimedean and $e\in E^+$, we call $(E,e)$ a \textit{positively pointed Archimedean vector lattice}.
	\end{definition}

	\section{The Order-Theoretic Square Operation on Pointed Vector Lattices}\label{S: square operation}
	
	We introduce the order-theoretic square operation on pointed vector lattices in this section and provide some of its properties, which will be useful to us in the next section.
	
	\begin{definition}\label{D: square closed}
		Given a pointed vector lattice $(E,e)$ and $f\in E$, if
		\[
		\underset{\lambda\in\R}{\sup}\{2\lambda f-\lambda^2e\}
		\]
		exists in $E$, then we write
		\[
		f^{[2]}:=\underset{\lambda\in\R}{\sup}\{2\lambda f-\lambda^2e\}.
		\]
		In this case, we say that $f^{[2]}$ \textit{exists} in $E$ and call $f^{[2]}$ the \textit{order-theoretic square} of $f$ with respect to $e$. Moreover, we set
		\[
		\E:=\left\{f\in E : f^{[2]}\ \text{exists in}\ E\right\}.
		\]
		Finally, we say that $(E,e)$	is \textit{square closed} if $\E=E$. 
	\end{definition}
	
	We stress that, in order to both avoid ambiguity and keep our notation as efficient as possible, we only use the notation $f^{[2]}$ and $\E$ is the context of a pointed vector lattice $(E,e)$. 
	
	\begin{notation}
		Given an $\ell$-algebra $E$ and $e\in E$, we will throughput this paper use $f^{[2]}$ to denote the order-theoretic square of $f$ with respect to $e$ in the pointed vector lattice $(E,e)$ and signify the algebraic square of $f$, as usual, by $f^2$.
	\end{notation}
	
	Our first result provides a wide class of examples of square closed pointed vector lattices.
	
	\begin{proposition}\label{P: phi-alg implies sc}
		If $E$ is an Archimedean $\Phi$-algebra with multiplicative unit $e$, then $(E,e)$ is square closed. Moreover, for every $f\in E$, we have $f^{[2]}=f^2$.
	\end{proposition}
	
	\begin{proof}
		Suppose $E$ is an Archimedean $\Phi$-algebra with multiplicative unit $e$. Let $f\in E$, and let $\lambda\in\R$. Then we have
		\[
		f^2-2\lambda f+\lambda^2e=(f-\lambda e)^2\geq 0.
		\]
		It follows that $f^2$ is an upper bound for the set $\{2\lambda f-\lambda^2e : \lambda\in\R\}$.
		
		Next suppose that $u\in E$ satisfies $u\geq 2\lambda f-\lambda^2e$ for every $\lambda\in\R$. Then we have
		\[
		\lambda^2e-2\lambda f+u\geq 0\qquad (\lambda\in\R).
		\]
		By \cite[Proposition~3.3]{HuidP}, we have that $(2f)^2\leq 4eu$, and hence $u\geq f^2$. Therefore, we have that $\underset{\lambda\in\R}{\sup}\{2\lambda f-\lambda^2e\}$ exists in $E$, and that
		\[
		f^2=\underset{\lambda\in\R}{\sup}\{2\lambda f-\lambda^2e\}.
		\]
		Hence $(E,e)$ is square closed, and $f^{[2]}=f^2$ for every $f\in E$.
	\end{proof}
	
	If $(E,e)$ is a positively pointed vector lattice, $f\in\E$, and $f$ is comparable with a scalar multiple of $e$, then the formula for $f^{[2]}$ in Definition~\ref{D: square closed} conveniently simplifies. We record this useful fact, which we will utilize throughout the paper, in our next proposition.
	
	\begin{proposition}\label{P: easy sup}
		Let $(E,e)$ be a positively pointed vector lattice. Fix $f\in E$ and $s,t\in\R$. The following hold.
		\begin{itemize}
			\item[(i)] If $f\leq te$, then $f\in\E$ if and only if $\underset{\lambda\in(-\infty, t]}{\sup}\{2\lambda f-\lambda^2e\}$ exists in $E$. In this case, we have
			\[
			f^{[2]}=\underset{\lambda\in(-\infty, t]}{\sup}\{2\lambda f-\lambda^2e\}.
			\]
			\item[(ii)] If $f\geq se$, then $f\in\E$ if and only if $\underset{\lambda\in[s,\infty)}{\sup}\{2\lambda f-\lambda^2e\}$ exists in $E$. In this case, we get
			\[
			f^{[2]}=\underset{\lambda\in[s, \infty)}{\sup}\{2\lambda f-\lambda^2e\}.
			\]
			\item[(iii)] If $f\in E^+$, then $f\in\E$ if and only if $\underset{\lambda\in\R^+}{\sup}\{2\lambda f-\lambda^2e\}$ exists in $E$. In this case, we have
			\[
			f^{[2]}=\underset{\lambda\in\R^+}{\sup}\{2\lambda f-\lambda^2e\}.
			\]
			\item[(iv)] If $se\leq f\leq te$, then $f\in\E$ if and only if $\underset{\lambda\in[s,t]}{\sup}\{2\lambda f-\lambda^2e\}$ exists in $E$. In this case, we get
			\[
			f^{[2]}=\underset{\lambda\in[s,t]}{\sup}\{2\lambda f-\lambda^2e\}.
			\]
			\item[(v)] If $|f|\leq te$, then $f\in\E$ if and only if $\underset{\lambda\in[-t,t]}{\sup}\{2\lambda f-\lambda^2e\}$ exists in $E$. In this case, we have
			\[
			f^{[2]}=\underset{\lambda\in[-t,t]}{\sup}\{2\lambda f-\lambda^2e\}.
			\] 
		\end{itemize}
	\end{proposition}
	
	\begin{proof}
		(i) Let $t\in\R$. Assume that $f\leq te$, and let $\lambda\in\R$. We will show that there exists $\lambda_0\in(-\infty,t]$ such that
		\begin{equation}\label{eq: easy sup}
			2\lambda f-\lambda^2e\leq 	2\lambda_0 f-\lambda_0^2e. 
		\end{equation}
		To this end, first note that if $\lambda\leq t$, then simply taking $\lambda_0:=\lambda$ satisfies \eqref{eq: easy sup}.
		
		Suppose next that $\lambda>t$. In this case, set $\lambda_0:=t$. Since $f\leq te$ and $e\in E^+$ (as $(E,e)$ is positively pointed), we have
		\[
		2f\leq 2te\leq(\lambda+\lambda_0)e.
		\]
		Multiplying both sides of $2f\leq(\lambda+\lambda_0)e$ by the strictly positive real number $\lambda-\lambda_0$ yields
		\[
		2(\lambda-\lambda_0)f\leq(\lambda^2-\lambda_0^2)e,
		\]
		which is equivalent to
		\[
		2\lambda f-\lambda^2e\leq 2\lambda_0f-\lambda_0^2e.	
		\]
		Hence we have shown that there exits $\lambda_0\in(-\infty,t]$ for which \eqref{eq: easy sup} holds. It follows that $\underset{\lambda\in\R}{\sup}\{2\lambda f-\lambda^2e\}$ exists in $E$ if and only if $\underset{\lambda\in(-\infty,t]}{\sup}\{2\lambda f-\lambda^2e\}$ exists in $E$, in which case we have
		\[
		f^{[2]}=\underset{\lambda\in(-\infty,t]}{\sup}\{2\lambda f-\lambda^2e\}.
		\]

		The proof of (ii) is similar to (i), and (iii) is simply a special case of (ii). Part (iv) is an immediate consequence of (i) and (ii), and clearly, (v) is a special case of (iv).
	\end{proof}
	
	Next, given a pointed vector lattice $(E,e)$, we provide some closure properties of $\E$, which will prove to be useful throughout the rest of the paper.
	
	\begin{proposition}\label{P: closure props} Let $(E,e)$ be a pointed vector lattice. The following hold.
		\begin{itemize}
			\item[(i)] If $f\in\E$ and $t\in\R\setminus\{0\}$, then $t f\in\E$ and $(tf)^{[2]}=t^2f^{[2]}$.
			\item[(ii)] If $f\in\E$ and $t\in\R$, then $f+te\in\E$ and $(f+te)^{[2]}=f^{[2]}+2tf+t^2e$.
		\end{itemize}
	\end{proposition}
	
	\begin{proof}
		(i)  Let $f\in\E$. If $t\in\R$ and $t\neq0$, then, for any $\lambda\in\R$, we have
		\[
		2\lambda t f-\lambda^2e=t^2\left(2\dfrac{\lambda}{t}f-\left(\dfrac{\lambda}{t}\right)^2e\right).
		\]
		Hence we get
		\[
		\underset{\lambda\in\mathbb{R}}{\sup}\{2\lambda tf-\lambda^2e\}=t^2\, \underset{\lambda\in\mathbb{R}}{\sup}\left\{2\dfrac{\lambda}{t}f-\left(\dfrac{\lambda}{t}\right)^2e\right\}=t^2f^{[2]}.
		\]
		
		(ii) Let $f\in\E$ and $t\in\R$. 
		Consider $\lambda\in\R$, and set $\theta:=\lambda+t$. We have \begin{align*}
			2\lambda f-\lambda^2e+2t f+t^2e&=2(\lambda+t) f-(\lambda^2-t^2)e\\
			&=2\theta f-\bigl((\theta-t)^2-t^2\bigr)e\\
			&=2\theta f-(\theta^2-2\theta t)e\\
			&=2\theta(f+te)-\theta^2e.
		\end{align*}
		Thus we get
		\begin{align*}
			(f+te)^{[2]}&=\underset{\theta\in\R}{\sup}\{2\theta(f+te)-\theta^2e\}\\
			&=\underset{\lambda\in\R}{\sup}\{2\lambda f-\lambda^2e+2tf+t^2e\}\\
			&=\underset{\lambda\in\R}{\sup}\{2\lambda f-\lambda^2e\}+2tf+t^2e\\
			&=f^{[2]}+2tf+t^2e.
		\end{align*}
	\end{proof}
	
	The following result implies that Proposition~\ref{P: closure props}(i) also holds for $t=0$ in positively pointed vector lattices.
	
	\begin{proposition}\label{P: 0^2=0 iff e pos}
		Let $(E,e)$ be a pointed vector lattice. We have $0^{[2]}=0$ if and only if $(E,e)$ is positively pointed.
	\end{proposition}
	
	\begin{proof}
		First suppose that $0^{[2]}=0$. Then we have
		\[
		\underset{\lambda\in\R}{\sup}\{-\lambda^2e\}=0.
		\]
		Thus, for every $\lambda\in\R$, we have $-\lambda^2e\leq 0$. Hence we get $e\in E^+$.
		
		Second, assume that $e\in E^+$. Then we have $-\lambda^2e\leq 0$ for every $\lambda\in \R$. Suppose next that $u\in E$ satisfies $-\lambda^2e\leq u$ for every $\lambda\in\R$. Then setting $\lambda=0$ yields $0\leq u$. It follows that
		\[
		0=\underset{\lambda\in\R}{\sup}\{-\lambda^2e\}=\underset{\lambda\in\R}{\sup}\{2\lambda0-\lambda^2e\}=:0^{[2]}.
		\]
	\end{proof}
	
	As a corollary to Propositions~\ref{P: closure props} and \ref{P: 0^2=0 iff e pos}, we have the following result.
	
	\begin{corollary}\label{C: Span e contained}
		If $(E,e)$ is a positively pointed vector lattice, then $e^{[2]}=e$, and moreover, $\mathrm{Span}\{e\}\subseteq\E$.
	\end{corollary}
	
	\begin{proof}
		Let $(E,e)$ be a positivity pointed vector lattice. Fix $\lambda\in\R$. Since $2\lambda-\lambda^2\leq 1$ and $e\in E^+$, we have
		\[
		2\lambda e-\lambda^2e=(2\lambda-\lambda^2)e\leq e.
		\]
		Thus $e$ is an upper bound for the set $\{2\lambda e-\lambda^2e : \lambda\in\R\}$. Next assume that $u$ is an upper bound for $\{2\lambda e-\lambda^2e : \lambda\in\R\}$. Then, setting $\lambda=1$, we get $u\geq 2e-e=e$. Hence $e$ is the least upper bound of $\{2\lambda e-\lambda^2e : \lambda\in\R\}$; that is $e\in\E$ with $e^{[2]}=e$.
		
		From the above argument and Proposition~\ref{P: 0^2=0 iff e pos}, we have $te\in\E$ for $t\in\{0,1\}$. From Proposition\ref{P: closure props}(ii), we get $te\in\E$ for all $t\in\R$. It follows that $\mathrm{Span}\{e\}\subseteq E^{[2]}$.
	\end{proof}
	
	Given a positively pointed vector lattice $(E,e)$, it is possible that $\mathrm{Span}\{e\}=\E$, even if $E$ is infinite-dimensional, as the following example demonstrates.
	
	\begin{example}\label{ex: pp}
		Fix $n\in\N\cup\{0\}$, and consider the Archimedean vector sublattice $PP_n[0,1]$ of $C[0,1]$ consisting of all piecewise polynomials $f\colon[0,1]\to\R$ of degree at most $n$. Let $\mathbf{1}$ denote the constant function on $[0,1]$ taking the value 1. Then $\bigl(PP_n[0,1], \mathbf{1}\bigr)$ is a positively pointed Archimedean vector lattice, and
		\[
		\bigl(PP_n[0,1]\bigr)^{[2]}=PP_{\lfloor\frac{n}{2}\rfloor}[0,1].
		\]
		In particular, we have
		\[
		(PP_1[0,1])^{[2]}=\mathrm{Span}\{\mathbf{1}\}.
		\]
		
		Indeed, if $n=0$, then $PP_{n}[0,1]$ is the space of all constant functions on $[0,1]$, and is therefore a $\Phi$-algebra. From Proposition~\ref{P: phi-alg implies sc}, we have $(PP_0[0,1])^{[2]}=PP_0[0,1]$. Next consider $n\in\N$. If $f\in PP_{\lfloor\frac{n}{2}\rfloor}[0,1]$, then $f$ can be expressed in the form
		\[
		f(x)=\sum_{k=1}^{m}f_k(x)\chi_{A_k}(x),
		\]
		where, for each $k\in\{1,...,m\}$, $f_k$ is a polynomial of degree at most $\lfloor\frac{n}{2}\rfloor$, $A_k$ is a subinterval of $[0,1]$, $A_i\cap A_j=\varnothing\quad (i\neq j)$, and $\bigcup_{k=1}^m A_k=[0,1]$. Since $C[0,1]$ is a $\Phi$-algebra with multiplicative unit $\mathbf{1}$, $(C[0,1],\mathbf{1})$ is square closed by Proposition~\ref{P: phi-alg implies sc}. Hence $f^{[2]}$ exists in $C[0,1]$, and in $C[0,1]$ we clearly have
		\[
		f^{[2]}=\sum_{k=1}^{m}f_k^2\chi_{A_k}.
		\]
		Since $f\in PP_{\lfloor\frac{n}{2}\rfloor}[0,1]$, we have $f_k^2\in PP_{n}[0,1]$ for all $k\in\{1,...,m\}$. It follows that $f^{[2]}\in (PP_{n}[0,1])^{[2]}$.
		
		Assume next that $f\in \bigl(PP_n[0,1]\bigr)^{[2]}\setminus PP_{\lfloor\frac{n}{2}\rfloor}[0,1]$. Suppose $u\in PP_n[0,1]$ is such that $2\lambda f-\lambda^2\mathbf{1}\leq u$ holds for all $\lambda\in\R$. Then there exists $s,t\in[0,1]$ with $s<t$ and polynomials $p,q\colon(s,t)\to\R$, with the degree of $p$ greater than $\lfloor\frac{n}{2}\rfloor$ and at most $n$, and the degree of $q$ at most $n$, for which $f(x)=p(x)$ and $u(x)=q(x)\ \bigl(x\in(s,t)\bigr)$. Therefore, for $x\in(s,t)$ and $\lambda\in\R$, we have $2\lambda p(x)-\lambda^2<q(x)$. Taking the supremum over every $\lambda\in\R$, we obtain $\bigl(p(x)\bigr)^2\leq q(x)\quad \bigl(x\in(s,t)\bigr)$. This inequality implies that the degree of $q$ is larger than $n$, a contradiction. Hence $f^{[2]}$ does not exist in $PP_n[0,1]$.
	\end{example} 
	
	Recall that if $E$ is an associative algebra, then the algebraic square can encode some important information about $E$. Indeed, for a vector lattice and commutative associative algebra $E$, we have that $E$ is an $\ell$-algebra if and only if
	\begin{equation}\label{eq: Riesz alg}
		f,g\in E^+\implies f^2+g^2\leq(f+g)^2.
	\end{equation}
	For further examples, given an $f$-algebra $E$ and $f,g\in E^+$, we have
	\begin{equation}\label{eq: f-alg 1}
		f^2\wedge g^2>0\implies f\wedge g>0, 
	\end{equation}
	and
	\begin{equation}\label{eq: f-alg 2}
		(f+g)^2\neq f^2+g^2\implies f\wedge g>0.
	\end{equation}
	
	Our next result shows that the order-theoretic square also respects the properties \eqref{eq: Riesz alg}, \eqref{eq: f-alg 1}, and \eqref{eq: f-alg 2} above in square closed positively pointed vector lattices. It also shows an additional convenient property of the order theoretic square which we will employ in the next section.
	
	\begin{proposition}\label{P: basic props}
		Let $(E,e)$ be a square closed positively pointed vector lattice. The following are true.
		\begin{itemize}
			\item[(i)]  For all $f,g\in E^+$, we have that $f^{[2]}+g^{[2]}\leq(f+g)^{[2]}$.
			\item[(ii)] If $f,g\in E$ satisfy $f\wedge g=0$, then $f^{[2]}\wedge g^{[2]}=0$.
			\item[(iii)] If $f,g\in E$ satisfy $f\wedge g=0$, then $(f+g)^{[2]}=f^{[2]}+g^{[2]}$.
			\item[(iv)] For every $f\in E$, we have $f^{[2]}=|f|^{[2]}$.
		\end{itemize}
	\end{proposition}
	
	\begin{proof}
		(i) Let $f,g\in E^+$, let $\alpha,\beta\in\R^+$ be arbitrary, and set $\theta:=\sqrt{\alpha^2+\beta^2}$. Using Proposition~\ref{P: easy sup}(iii) in the second identity below, we obtain
		\begin{align*}
			2\alpha f-\alpha^2e+2\beta g-\beta^2e&=2(\alpha f+\beta g)-(\alpha^2+\beta^2)e\\
			&\leq 2\theta(f+g)-\theta^2e\\
			&\leq\underset{\lambda\in\R^+}{\sup}\{2\lambda(f+g)-\lambda^2e\}\\
			&=(f+g)^{[2]}.
		\end{align*}
		Taking suprema over all $\alpha,\beta\in\R^+$ and utilizing Proposition\ref{P: easy sup}(iii) twice more, we get
		\[
		f^{[2]}+g^{[2]}\leq(f+g)^{[2]}.
		\]
		
		(ii) Suppose $f,g\in E$ satisfy $f\wedge g=0$. Then $f,g\in E^+$. Moreover, we have
		\[
		f^{[2]}=\underset{\lambda\in\R}{\sup}\{2\lambda f-\lambda^2e\}\geq 2\cdot 0f-0^2e=0. 
		\]
		Similarly, $g^{[2]}\geq 0$. We thus have $f^{[2]}\wedge g^{[2]}\geq 0$.
		
		On the other hand, let $\alpha,\beta\in\R^+$ be arbitrary. Then we get
		\[
		(2\alpha f-\alpha^2e)\wedge(2\beta g-\beta^2e)\leq (2\alpha f)\wedge(2\beta g)\leq 2\max\{\alpha,\beta\}(f\wedge g)=0.
		\]
		Taking suprema over all $\alpha,\beta\in\R^+$ and employing Proposition~\ref{P: easy sup}(iii) twice again, we obtain $f^{[2]}\wedge g^{[2]}\leq 0$. We conclude that
		\[
		f^{[2]}\wedge g^{[2]}=0.
		\]
		
		(iii) Let $f,g\in E$ be such that $f\wedge g=0$. Then $f,g\in E^+$, so part (i) of this proposition implies that $(f+g)^{[2]}\geq f^{[2]}+g^{[2]}$. We will further show that $(f+g)^{[2]}\leq f^{[2]}+g^{[2]}$. Indeed, using Proposition~\ref{P: easy sup}(iii) in the first and fourth equalities below, and part (ii) of this proposition in the fifth identity below, we obtain
		\begin{align*}
			(f+g)^{[2]}&=\underset{\lambda\in\R^+}{\sup}\{2\lambda(f+g)-\lambda^2e\}\\
			&=\underset{\lambda\in\R^+}{\sup}\{2\lambda(f\vee g)-\lambda^2e\}\\
			&=\underset{\lambda\in\R^+}{\sup}\{(2\lambda f-\lambda^2e)\vee(2\lambda g-\lambda^2e)\}\\
			&\leq\underset{\alpha,\beta\in\R^+}{\sup}\{(2\alpha f-\alpha^2e)\vee(2\beta g-\beta^2e)\}\\
			&=f^{[2]}\vee g^{[2]}\\
			&=f^{[2]}+g^{[2]}. 
		\end{align*}
		
		(iv) Let $f\in E$. Then we have
		\begin{align*}
			f^{[2]}&=\underset{\lambda\in\R}{\sup}\{2\lambda f-\lambda^2e\}\\
			&=\underset{\lambda\in\R^+}{\sup}\{2\lambda f-\lambda^2e\}\vee\underset{\lambda\in\R^+}{\sup}\{2\lambda(-f)-\lambda^2e\}\\
			&\leq\underset{\lambda\in\R}{\sup}\{2\lambda|f|-\lambda^2e\}\\
			&=|f|^{[2]}.
		\end{align*}
		
		On the other hand, using Proposition~\ref{P: easy sup}(iii) in the first equality below, we obtain
		\begin{align*}
			|f|^{[2]}&=\underset{\lambda\in\R^+}{\sup}\{2\lambda|f|-\lambda^2e\}\\
			&=\underset{\lambda\in\R^+}{\sup}\{2\lambda(f\vee-f)-\lambda^2e\}\\
			&=\underset{\lambda\in\R^+}{\sup}\{(2\lambda f-\lambda^2e)\vee(2\lambda(-f)-\lambda^2e)\}\\
			&\leq\underset{\alpha,\beta\in\R}{\sup}\{(2\alpha f-\alpha^2e)\vee(2\beta(-f)-\beta^2e)\}\\
			&=\underset{\alpha\in\R}{\sup}\{2\alpha f-\alpha^2e\}\vee\underset{\beta\in\R}{\sup}\{2\beta(-f)-\beta^2e\}\\
			&=\underset{\alpha\in\R}{\sup}\{2\alpha f-\alpha^2e\}\vee\underset{\beta\in\R}{\sup}\{2(-\beta)f-(-\beta)^2e\}\\
			&=\underset{\alpha\in\R}{\sup}\{2\alpha f-\alpha^2e\}\vee\underset{\gamma\in\R}{\sup}\{2\gamma f-\gamma^2e\}\\
			&=f^{[2]}.
		\end{align*}
		Therefore, we obtain $f^{[2]}=|f|^{[2]}$.
	\end{proof}
	
	We close this section with the following order continuity property of the map $f\mapsto f^{[2]}$ in a square closed positively pointed vector lattice. This proposition will be needed for our main result Theorem~\ref{T: main}.
	
	\begin{proposition}\label{P: ord cts} Let $(E,e)$ be a square closed positively pointed vector lattice. If $\{x_n\}_{n\in\N}$ is a sequence in $E^+$ and $f\in E^+$ are such that $x_n\uparrow f$, then $\left\{x_n^{[2]}\right\}_{n\in\N}$ is a sequence in $E^+$, and $x_n^{[2]}\uparrow f^{[2]}$.
	\end{proposition}
	
	\begin{proof}
		Let $f\in E^+$, and suppose $\{x_n\}_{n\in\N}$ is a sequence in $E^+$ such that $x_n\uparrow f$. We claim that $0\leq x_n^{[2]}\uparrow f^{[2]}$. To verify this assertion, fix $n\in\N$. We have
		\[
		x_n^{[2]}=\underset{\lambda\in\R}{\sup}\{2\lambda x_n-\lambda^2e\}\geq2\cdot 0x_n-0^2e=0.
		\]
		Therefore, using Proposition~\ref{P: easy sup}(iii) twice, we have
		\[
		x_{n+1}^{[2]}=\underset{\lambda\in\R^+}{\sup}\{2\lambda x_{n+1}-\lambda^2e\}\geq\underset{\lambda\in\R^+}{\sup}\{2\lambda x_n-\lambda^2e\}=x_n^{[2]}.
		\]
		Thus $\left\{x_n^{[2]}\right\}_{n\in\N}$ is an increasing sequence in $E^+$.
		
		Next let $\theta\in\mathbb{R}^+$, and fix $n\in\N$. From the assumption that $x_n\leq f$, and from Proposition~\ref{P: easy sup}(iii), we obtain 
		\[
		2\theta x_n-\theta^2e\leq 2\theta f-\theta^2e\leq\underset{\lambda\in\R^+}{\sup}\{2\lambda f-\lambda^2e\}=f^{[2]}.
		\]
		We thus have
		\[
		x_n^{[2]}=\underset{\theta\in\R^+}{\sup}\{2\theta x_n-\theta^2e\}\leq f^{[2]}.
		\]
		
		Assume next that $u\in E$ satisfies $x_n^{[2]}\leq u$ for all $n\in\N$. It follows that $u\geq 0$. Let $\theta\in\R^+$. If $\theta=0$, we trivially have
		\[
		2\theta f - \theta^2e\leq u.
		\]
		Suppose $\theta>0$. From Proposition~\ref{P: easy sup}(iii), we get
		\begin{align*}
			2\theta x_n-\theta^2e&\leq\underset{\lambda\in\R^+}{\sup}\{2\lambda x_n-\lambda^2e\}\\
			&=x_n^{[2]}\\
			&\leq u.
		\end{align*}
		Thus
		\[
		x_n\leq\dfrac{\theta^2e+u}{2\theta},
		\]
		and hence taking the supremum over all $n\in\N$ gives us
		\[
		f\leq\dfrac{\theta^2e+u}{2\theta}.
		\]
		It follows that
		\[
		2\theta f-\theta^2e\leq u.
		\]
		Finally, utilizing Proposition~\ref{P: easy sup}(iii) once more, we obtain
		\[
		f^{[2]}=\underset{\theta\in\R^+}{\sup}\{2\theta f-\theta^2e\}\leq u.
		\]
		We conclude that $x_n^{[2]}\uparrow f^{[2]}$.
	\end{proof}
	
	\section{Multiplication}
	
	We define a multiplication on square closed pointed Archimedean vector lattices as follows.
	
	\begin{definition}\label{D: multn}
		Suppose $(E,e)$ is a square closed pointed Archimedean vector lattice. We define
		\[
		fg:=\dfrac{1}{2}\left((f+g)^{[2]}-\left(f^{[2]}+g^{[2]}\right)\right)\qquad (f,g\in E).
		\]
	\end{definition}
	
	While Definition~\ref{D: multn} is arguably a natural avenue for trying to define a product on \textit{any} square closed pointed vector lattice, it is a definition that notoriously makes it difficult to prove the distributive and associative properties (see \cite[Section 34]{LilZan}), as well as other axioms of a $\Phi$-algebra multiplication. For this reason, we only consider Archimedean vector lattices in this section, where more tools are available.
	
	Theorem~\ref{T: weak order unit} below shows that, in the class of positively pointed Archimedean vector lattices, only the ones with a weak order unit can be square closed. In order to verify this claim, we require the following lemma.
	
	\begin{lemma}\label{L: 0^2 exists iff e pos}
		Let $(E,e)$ be a pointed Archimedean vector lattice. If $0\in\E$, then $0^{[2]}=0$.
	\end{lemma}
	
	\begin{proof}
		Assume that $0\in\E$. Then $u:=\underset{\lambda\in\R}{\sup}\{2\lambda0-\lambda^2e\}=\underset{\lambda\in\R}{\sup}\{-\lambda^2e\}\in E^+$. Hence, for every $\lambda\in\R$, we have $\lambda^2(-e)\leq u$. Since $E$ is Archimedean, we have $-e\leq 0$, and so $e\in E^+$. In other words, $(E,e)$ is positively pointed. By Proposition~\ref{P: 0^2=0 iff e pos}, we have $0^{[2]}=0$.
	\end{proof}
	
	\begin{theorem}\label{T: weak order unit}
		Let $(E,e)$ be a square closed pointed Archimedean vector lattice. Then $e$ is a weak order unit of $E$.
	\end{theorem}
	
	\begin{proof}
		First note that, since $(E,e)$ is square closed, we have $0\in\E$. By Lemma~\ref{L: 0^2 exists iff e pos}, we have that $0^{[2]}=0$, and so we get $e\in E^+$ from Proposition~\ref{P: 0^2=0 iff e pos}. 
		
		Suppose next that $f\in E^+$ and that $f\perp e$. Let $\lambda\in\R^+\setminus\{0\}$. Then $2\lambda f-\lambda^2e\leq f^{[2]}$,
		and so $f-\dfrac{\lambda}{2}e\leq \dfrac{1}{2\lambda}f^{[2]}$. Hence $\left(f-\dfrac{\lambda}{2}e\right)^+\leq\dfrac{1}{2\lambda}f^{[2]}$. Since $f\perp-\dfrac{\lambda}{2}e$, we have by \cite[Theorem~8.2(2)]{LilZan} that
		\[
		\left(f-\dfrac{\lambda}{2}e\right)^+=f^+ +(-\dfrac{\lambda}{2}e)^+=f^+.
		\]
		Thus we obtain $f^+\leq\dfrac{1}{2\lambda}f^{[2]}$. Since $\lambda>0$ is arbitrary and $E$ is Archimedean, we conclude that $f^+=0$. That $f^-=0$ can be shown is a symmetrical manner. Hence $f=0$, and thus $e$ is a weak order unit.
	\end{proof}
	
	Our next result, combined with the theory of small Riesz spaces \cite{BusvR}, will prove to be an extremely powerful tool for the remainder of the manuscript. Its proof is an approximation argument similar to what is found in the proof of \cite[Proposition~3.4]{AzBoBus}.
	
	\begin{theorem}\label{T: hom intergchange}
		Let $(E,e_1)$ and $(F,e_2)$ be pointed Archimedean vector lattices with $e_1$ and $e_2$ strong order units. Consider a vector lattice homomorphism $T\colon E\to F$ with the property that $T(e_1)=e_2$. For every $f\in\E$ such that $T(f)\in F^{[2]}$, we have
		\[
		T\left(f^{[2]}\right)=\bigl(T(f)\bigr)^{[2]}.
		\]
	\end{theorem}
	
	\begin{proof}
		Let $f\in\E$, assume $T(f)\in F^{[2]}$, and let $n\in\N$ be arbitrary. Since $e_1$ and $e_2$ are strong order units, there exists $M\in\R^+$ such that $|f|\leq Me_1$ and $|T(f)|\leq Me_2$. Fix $\lambda\in[-M,M]$. There exists $k_0\in\{-n,...,0, ...,n\}$ such that
		\[
		\left|\lambda-\dfrac{Mk_0}{n}\right|<\dfrac{M}{n}.
		\]
		Hence we obtain
		\begin{align*}
			2\lambda f-\lambda^2e_1-\bigvee_{k=-n}^n\left(2\dfrac{Mk}{n}f-\left(\dfrac{Mk}{n}\right)^2e_1\right)&\leq 2\lambda f-\lambda^2e_1-\left(2\dfrac{Mk_0}{n}f-\left(\dfrac{Mk_0}{n}\right)^2e_1\right)\\
			&=2\left(\lambda-\dfrac{Mk_0}{n}\right)f-\left(\lambda^2-\left(\dfrac{Mk_0}{n}\right)^2\right)e_1\\
			&\leq 2\left|\lambda-\dfrac{Mk_0}{n}\right||f|+\left|\lambda-\dfrac{Mk_0}{n}\right|\left|\lambda+\dfrac{Mk_0}{n}\right|e_1\\
			&\leq 2\dfrac{M}{n}|f|+\dfrac{M}{n}\cdot2Me_1\\
			&=\dfrac{2M}{n}\bigl(|f|+Me_1\bigr).
		\end{align*}
		Taking the supremum over all $\lambda\in[-M,M]$, and employing Proposition~\ref{P: easy sup}(v), we get
		\begin{align*}
			\left|f^{[2]}-\bigvee_{k=-n}^n\left(2\dfrac{Mk}{n}f-\left(\dfrac{Mk}{n}\right)^2e_1\right)\right|&=f^{[2]}-\bigvee_{k=-n}^n\left(2\dfrac{Mk}{n}f-\left(\dfrac{Mk}{n}\right)^2e_1\right)\\
			&\leq\dfrac{2M}{n}\bigl(|f|+Me_1\bigr).
		\end{align*}
		
		Likewise, we have
		\[
		\left|\bigl(T(f)\bigr)^{[2]}-\bigvee_{k=-n}^n\left(2\dfrac{Mk}{n}T(f)-\left(\dfrac{Mk}{n}\right)^2e_2\right)\right|\leq \dfrac{2M}{n}\bigl(|T(f)|+Me_2\bigr).
		\]
		Therefore, since $T$ is a vector lattice homomorphism, and $T(e_1)=e_2$, we get
		\begin{align*}
			\left|T\left(f^{[2]}\right)-\bigl(T(f)\bigr)^{[2]}\right|&\leq\left|T\left(f^{[2]}\right)-T\left(\bigvee_{k=-n}^n\left(2\dfrac{Mk}{n}f-\left(\dfrac{Mk}{n}\right)^2e_1\right)\right)\right|\\
			&+\left|\bigl(T(f)\bigr)^{[2]}-T\left(\bigvee_{k=-n}^n\left(2\dfrac{Mk}{n}f-\left(\dfrac{Mk}{n}\right)^2e_1\right)\right)\right|\\
			&=\left|T\left(f^{[2]}-\bigvee_{k=-n}^n\left(2\dfrac{Mk}{n}f-\left(\dfrac{Mk}{n}\right)^2e_1\right)\right)\right|\\
			&+\left|\bigl(T(f)\bigr)^{[2]}-\bigvee_{k=-n}^n\left(2\dfrac{Mk}{n}T(f)-\left(\dfrac{Mk}{n}\right)^2e_2\right)\right|\\
			&\leq T\left(\dfrac{2M}{n}\bigl(|f|+Me_1\bigr)\right)+\dfrac{2M}{n}\bigl(|T(f)|+Me_2\bigr)\\
			&=\dfrac{4M}{n}\bigl(|T(f)|+Me_2\bigr).
		\end{align*}
		Since $n\in\N$ was arbitrary and $F$ is Archimedean, we have $T\left(f^{[2]}\right)=\bigl(T(f)\bigr)^{[2]}$.
	\end{proof}
	
	As an immediate corollary, we obtain the following result regarding the multiplication defined in Definition~\ref{D: multn}.
	
	\begin{corollary}\label{C: T mult've}
		Let $(E,e_1)$ and $(F,e_2)$ be square closed pointed Archimedean vector lattices with $e_1$ and $e_2$ strong order units, and let $T\colon E\to F$ be a vector lattice homomorphism such that $T(e_1)=e_2$. Then for every $f,g\in E$, we have
		\[
		T(fg)=T(f)T(g).
		\]
	\end{corollary} 
	
	For another corollary to Theorem~\ref{T: hom intergchange}, we get the following result, which we will freely use throughout the remainder of the manuscript.
	
	\begin{corollary}\label{L: subspace invariant}
		Let $E_0$ and $E$ be Archimedean vector lattices with $E_0$ a vector sublattice of $E$. Fix $e\in E_0^+$, and suppose that $(E_0,e)$ and $(E,e)$ are both positively pointed Archimedean vector lattices with respect to $e$. Assume that $(E,e)$ is square closed, and that $e$ is a strong order unit of $E$. If $f\in\E_0$, then the value of $f^{[2]}$ is the same, whether its corresponding supremum is calculated in $E_0$ or $E$.
	\end{corollary}
	
	\begin{proof}
		Let $f\in\E_0$, and let $f^{[2]}$ be the order-theoretic square of $f$ with respect to $e$, where the associated supremum is calculated in $E_0$. Applying Theorem~\ref{T: hom intergchange} to the inclusion map $I\colon E_0\to E$, we obtain
		$I\left(f^{[2]}\right)=\bigl(I(f)\bigr)^{[2]}$.
	\end{proof}
	
	We will utilize the following notation and lemma below in the proof of our main result Theorem~\ref{T: main}.
	
	\begin{notation}
		Given a pointed vector lattice $(E,e)$, we denote by $I_e$ the ideal of $E$ generated by $e$.
	\end{notation}
	
	\begin{lemma}\label{L: I_e is sc}
		If $(E,e)$ is a square closed pointed Archimedean vector lattice, then $(I_e, e)$ is a square closed pointed Archimedean vector lattice.
	\end{lemma}
	
	\begin{proof}
		Let $(E,e)$ be a square closed pointed Archimedean vector lattice. We only need to prove that $(I_e, e)$ is square closed. To this end, let $f\in I_e$. Since $E$ is square closed, the map $f\mapsto f^{[2]}$ is defined as a map from $I_e$ to $E$. Thus we only need to show that $f^{[2]}\in I_e$. To this end, there exists $n\in\mathbb{N}$ such that $|f|\leq ne$, as it follows from Theorem~\ref{T: weak order unit} that $e\in E^+$. Therefore, by Proposition~\ref{P: easy sup}(v), we have that $f^{[2]}=\underset{\lambda\in[-n,n]}{\sup}\{2\lambda f-\lambda^2e\}$ holds in $E$. Furthermore, we have that $\lambda\in[-n,n]$ if and only if $-\lambda\in[-n,n]$. Hence, for any $\lambda\in[-n,n]$, we have
		\[
		2\lambda f-\lambda^2e\leq 2\lambda f =2\lambda f^++2(-\lambda)f^-\leq 2nf^++2nf^-=2n|f|\leq 2n^2e. 
		\]
		Taking the supremum of $\{2\lambda f-\lambda^2e : \lambda\in[-n,n]\}$ and utilizing Proposition~\ref{P: easy sup}(v) again, we get $f^{[2]}\leq 2n^2e$. Since $f^{[2]}\in E^+$, we have $f^{[2]}\in I_e$.
	\end{proof}
	
	The following notation will also be employed in the proof of Theorem~\ref{T: main}.
	
	\begin{notation}
		Given a pointed vector lattice $(E,e)$ and a sequence $\{f_n\}_{n\in\N}$ in $E$, we write $f_n\to f$ and $\lim\limits_{n\to\infty}f_n=f$ to indicate that $\{f_n\}_{n\in\N}$ converges in order to $f\in E$.
	\end{notation}
	
	We are ready to present the main result of the paper.
	
	\begin{theorem}\label{T: main}
		A pointed Archimedean vector lattice $(E,e)$ is square closed if and only if each of the following hold.
		\begin{itemize}
			\item[(i)] $f(g+h)=fg+fh\quad (f,g,h\in E)$.
			\item[(ii)] $(f+g)h=fh+gh\quad (f,g,h\in E)$.
			\item[(iii)] $(fg)h=f(gh)\quad (f,g,h\in E)$.
			\item[(iv)] $(\alpha f)(\beta g)=(\alpha\beta)fg\quad (f,g\in E, \alpha,\beta\in\R)$.
			\item[(v)] If $f,g\in E^+$, then $fg\in E^+$.
			\item[(vi)] If $f,g\in E$ satisfy $f\wedge g=0$, and $c\in E^+$, then $(cf)\wedge g=0$.
			\item[(vii)] If $f,g\in E$ are such that $f\wedge g=0$, and $c\in E^+$, then $(fc)\wedge g=0$.
			\item[(viii)] $fe=f\quad (f\in E)$.
			\item[(ix)] $ef=f\quad (f\in E)$.
		\end{itemize}
		In other words, an Archimedean vector lattice $E$ is a $\Phi$-algebra with multiplicative unit $e$  if and only if $(E,e)$ is square  closed. Furthermore, we have
		\[
		f^2=f^{[2]}\quad (f\in E).
		\]
	\end{theorem}
	
	\begin{proof}	
		Let $(E,e)$ be a pointed Archimedean vector lattice. We know from Proposition~\ref{P: phi-alg implies sc} that, if $E$ is a $\Phi$-algebra with multiplicative unit $e$, then $(E,e)$ is square closed.
		
		In proving the converse, it is clear from Definition~\ref{D: multn} that	$fg=gf$ holds for all $f,g\in E$. Hence, once we prove (i), (vi), and (viii), we will immediately also have (ii), (vii), and (ix). Thus we only need to prove (i), (iii)-(vi), and (viii).	Moreover, since the proofs of (iii), (iv), and (vi) are similar to the proof of (i), we, for brevity, omit much of their details.
		
		To this end, suppose $(E,e)$ is square closed. Note that, by Lemma~\ref{L: I_e is sc}, $(I_e, e)$ is also a square closed pointed Archimedean vector lattice. By Theorem~\ref{T: weak order unit}, $e$ is a weak order unit of $E$. Thus $(I_e,e)$ and $(E,e)$ are both positively pointed.
		
		(i) First let $f,g,h\in I_e$.
		
		Let $C$ denote the vector sublattice of $I_e$ generated by
		\[
		\{f,g,h,e,f(g+h), fg, fh\}.
		\]
		Let $\omega\colon C\to\mathbb{R}$ be a vector lattice homomorphism for which $\omega(e)=1$. Using Corollary~\ref{C: T mult've} repeatedly, we get
		\begin{align*}
			\omega\left(f(g+h)\right)&=\omega(f)\omega(g+h)\\
			&=\omega(f)\bigl(\omega(g)+\omega(h)\bigr)\\
			&=\omega(f)\omega(g)+\omega(f)\omega(h)\\
			&=\omega(fg)+\omega(fh)\\
			&=\omega\bigl(fg+fh\bigr).
		\end{align*}
		Since the set of all vector lattice homomorphisms $\omega\colon C\to\R$ for which $\omega(e)=1$ separate the points of $C$, \cite[Remark 1.2(ii) and Theorem~2.4(i)]{BusvR}, we have $f(g+h)=fg+fh$.
		
		Next consider, more generally, $f,g,h\in E$. Moreover, for every $x\in E$, and each $n\in\N$, define
		\[
		x_n:=x^+\wedge ne - x^-\wedge ne.
		\]
		Then $\{x_n\}_{n\in\N}$ is a sequence in $I_e$, and from Theorem~\ref{T: weak order unit}, we have that $x^+\wedge ne\uparrow x^+$ and $x^-\wedge ne\uparrow x^-$. Hence we have $x_n\to x$. By Proposition~\ref{P: ord cts}, we get $(x^+\wedge ne)^{[2]}\uparrow(x^+)^{[2]}$ and $(x^-\wedge ne)^{[2]}\uparrow(x^-)^{[2]}$. It thus follows, using \cite[Theorem~8.2(1)]{LilZan} in the third equality below, Proposition~\ref{P: basic props}(iii) in the fifth and sixth identities below, and Proposition~\ref{P: basic props}(iv) in the second and last equalities below, that, for any $n\in\N$,
		\begin{align*}
			x_n^{[2]}&=(x^+\wedge ne-x^-\wedge ne)^{[2]}\\
			&=|x^+\wedge ne-x^-\wedge ne|^{[2]}\\
			&=|x^+\wedge ne+x^-\wedge ne|^{[2]}\\
			&=(x^+\wedge ne+x^-\wedge ne)^{[2]}\\
			&=(x^+\wedge ne)^{[2]}+(x^-\wedge ne)^{[2]}\\
			&\uparrow(x^+)^{[2]}+(x^-)^{[2]}\\
			&=|x|^{[2]}\\
			&=x^{[2]}.
		\end{align*}
		In particular, we have $f_n\to f$, $g_n\to g$, $h_n\to h$, $f_n^{[2]}\uparrow f^{[2]}$, $g_n^{[2]}\uparrow g^{[2]}$, and $h_n^{[2]}\uparrow h^{[2]}$. Likewise, we get $g_n+h_n\to g+h$ and $(g_n+h_n)^{[2]}\uparrow(g+h)^{[2]}$. It is also true that $f_n+g_n+h_n\to f+g+h$ and $(f_n+g_n+h_n)^{[2]}\uparrow(f+g+h)^{[2]}$. Hence we get
		\[
		f_n^{[2]}+(g_n+h_n)^{[2]}\uparrow f^{[2]}+(g+h)^{[2]}.
		\]
		It follows that
		\begin{align*}
			f_n(g_n+h_n)&:=\dfrac{1}{2}\left((f_n+g_n+h_n)^{[2]}-\left(f_n^{[2]}+(g_n+h_n)^{[2]}\right)\right)\\
			&\to\dfrac{1}{2}\left((f+g+h) ^{[2]}-\left(f^{[2]}+(g+h)^{[2]}\right)\right)\\
			&=:f(g+h).
		\end{align*}	
		Similarly, we have $f_ng_n\to fg$ and $f_nh_n\to fh$. Furthermore, it follows from Definition~\ref{D: multn} and Lemma~\ref{L: I_e is sc} that $\{f_n(g_n+h_n)\}_{n\in\N}$, $\{f_ng_n\}_{n\in\N}$, and $\{f_nh_n\}_{n\in\N}$ are sequences in $I_e$. Thus we have already proved that for all $n\in\N$, we have $f_n(g_n+h_n)=f_ng_n+f_nh_n$. We thus obtain
		\begin{align*}
			f(g+h)&=\lim\limits_{n\to\infty}f_n(g_n+h_n)\\
			&=\lim\limits_{n\to\infty}(f_ng_n+f_nh_n)\\
			&=\lim\limits_{n\to\infty}f_ng_n+\lim\limits_{n\to\infty}f_nh_n\\
			&=fg+fh.
		\end{align*}

		(iii) Let $f,g,h\in I_e$, and let $C$ be the vector sublattice of $I_e$ generated by
		\[
		\{f,g,h,e,fg, gh, (fg)h, f(gh)\}.
		\]
		Let $\omega\colon C\to\mathbb{R}$ be a vector lattice homomorphism with $\omega(e)=1$. Using Corollary~\ref{C: T mult've} several times, we get $\omega\bigl((fg)h\bigr)=\omega(fg)\omega(h)=\omega(f)\omega(g)\omega(h)=\omega(f)\omega(gh)=\omega\bigl(f(gh)\bigr)$.
		It follows from \cite[Remark 1.2(ii) and Theorem~2.4(i)]{BusvR} that $(fg)h=f(gh)$. By approximating elements of $E$ with sequences in $I_e$, precisely as done in (i), we obtain the desired result.
		
		(iv) Fix $f,g\in I_e$ and $\alpha,\beta\in\R$. Let $C$ be the vector sublattice of $I_e$ generated by
		\[
		\{f,g,e,fg,(\alpha f)(\beta g)\}.
		\]
		Consider a vector lattice homomorphism $\omega\colon C\to\mathbb{R}$ such that $\omega(e)=1$. Using Corollary~\ref{C: T mult've} repeatedly, we get $\omega\bigl((\alpha f)(\beta g)\bigr)=\omega(\alpha f)\omega(\beta g)=\alpha\beta\omega(f)\omega(g)=\alpha\beta\omega(fg)$. From \cite[Remark 1.2(ii) and Theorem~2.4(i)]{BusvR}, we get $(\alpha f)(\beta g)=\alpha\beta fg$. The result follows from approximating elements of $E$ by sequences in $I_e$, exactly as done in (i).
		
		(v) Let $f,g\in E^+$. By Proposition~\ref{P: basic props}(i), we have $f^{[2]}+g^{[2]}\leq(f+g)^{[2]}$. Then from Definition~\ref{D: multn} we get
		\begin{align*}
			fg:=\dfrac{1}{2}\left((f+g)^{[2]}-\left(f^{[2]}+g^{[2]}\right)\right)\geq 0.
		\end{align*}
		
		(vi) Let $f,g\in I_e^+$ satisfy $f\wedge g=0$, and let $c\in I_e^+$. Suppose $C$ is the vector sublattice of $I_e$ generated by $\{f,g,c,e,cf\}$.
		Let $\omega\colon C\to\mathbb{R}$ be a vector lattice homomorphism for which $\omega(e)=1$. Using Corollary~\ref{C: T mult've}, we get $\omega\bigl((cf)\wedge g\bigr)=\omega(cf)\wedge\omega(g)=\bigl(\omega(c)\omega(f)\bigr)\wedge \omega(g)$. Thus, if $0\leq\omega(c)\leq 1$, then 
		$\omega\bigl((cf)\wedge g\bigr)=\bigl(\omega(c)\omega(f)\bigr)\wedge \omega(g)\leq \omega(f)\wedge\omega(g)=\omega(f\wedge g)=0$, and if $\omega(c)>1$, we have
		$\omega\bigl((cf)\wedge g\bigr)=\bigl(\omega(c)\omega(f)\bigr)\wedge\omega(g)\leq \omega(c)\bigl(\omega(f\wedge g)\bigr)=0$.
		It follows from \cite[Remark 1.2(ii) and Theorem~2.4(i)]{BusvR} that $(cf)\wedge g=0$. By approximating elements of $E$ with sequences in $I_e$, precisely as done in (i), we obtain the desired result.
		
		(viii) Let $f\in E$. Using Proposition~\ref{P: closure props}(ii) and Corollary~\ref{C: Span e contained}, we have
		\begin{align*}
			fe&:=\dfrac{1}{2}\bigl((f+e)^{[2]}-(f^{[2]}+e^{[2]})\bigr)\\
			&=\dfrac{1}{2}(f^{[2]}+2f+e-f^{[2]}-e)\\
			&=f.
		\end{align*}
		
		We have thus shown that $E$ is a $\Phi$-algebra with multiplicative unit $e$.
		
		Finally, for $f\in E$, we have via Definition~\ref{D: multn} and from Proposition~\ref{P: closure props}(i) that
		\begin{align*}
			f^2:=\dfrac{1}{2}\left((2f)^{[2]}-2f^{[2]}\right)=\dfrac{1}{2}\left(4f^{[2]}-2f^{[2]}\right)=f^{[2]}.
		\end{align*}
	\end{proof}
	
	As a consequence of Theorem~\ref{T: main}, we obtain a purely order-theoretic depiction of Archimedean semiprime $f$-algebras in Theorem~\ref{T: semiprime f-alg}. The result involves the following novel concept.
	
	\begin{definition}\label{D: pseudo square closed}
		We say that a vector lattice $E$ is \textit{pseudo square closed} if there exists a square closed pointed Archimedean vector lattice $(F,e)$ and an injective vector lattice homomorphism $\phi\colon E\to F$ such that, for every $f\in E$, we have $\underset{\lambda\in\R}{\sup}\{2\lambda\phi(f)-\lambda^2e\}\in\phi(E)$. Under these circumstances, we also say that $E$ is \textit{pseudo square closed by} $(F,e)$. 
		
	\end{definition}
	
	\begin{theorem}\label{T: semiprime f-alg}
		Let $E$ be an Archimedean vector lattice. Then $E$ is a semiprime $f$-algebra if and only if $E$ is pseudo square closed. In this case, if $E$ is pseudo square closed by $(F,e)$, then the $f$-algebra multiplication of $E$ satisfies
		\[
		\phi(f^2)=\underset{\lambda\in\R}{\sup}\{2\lambda\phi(f)-\lambda^2e\}\in\phi(E)\quad (f\in E)
		\]
		and
		\[
		\phi(fg)=\dfrac{1}{2}\left(\underset{\lambda\in\R}{\sup}\bigl\{2\lambda\phi(f+g)-\lambda^2e\bigr\}-\underset{\alpha,\beta\in\R}{\sup}\bigl\{2\alpha \phi(f)+2\beta \phi(g)-(\alpha^2+\beta^2)e\bigr\}\right)\ (f,g\in E).
		\]
	\end{theorem}
	
	\begin{proof}
		Assume first that $E$ is a semiprime $f$-algebra. Then there exists an injective multiplicative vector lattice homomorphism $\phi\colon E\to\mathrm{Orth}(E)$ \cite[Proposition~12.1]{dP}. Since $E$ is Archimedean, $\mathrm{Orth}(E)$ is an Archimedean $\Phi$-algebra with the identity map $I\colon E\to E$ as multiplicative unit \cite[Theorem~9.4]{dP}. By Theorem~\ref{T: main}, $(\mathrm{Orth}(E), I)$ is square closed and $\pi^2=\pi^{[2]}$ holds for every $\pi\in\mathrm{Orth}(E)$. Therefore, since $E$ is an $f$-algebra and $\phi$ is multiplicative, for any $f\in E$ we have
		\[
		\underset{\lambda\in\R}{\sup}\{2\lambda\phi(f)-\lambda^2I\}=\bigl(\phi(f)\bigr)^{[2]}=\bigl(\phi(f)\bigr)^2=\phi(f^2)\in\phi(E).
		\]
		Hence $E$ is pseudo square closed.
		
		Next suppose $E$ is pseudo square closed. Then there exists a square closed pointed Archimedean vector lattice $(F,e)$ and an injective vector lattice homomorphism $\phi\colon E\to F$, such that, for every $f\in E$,
		\[
		\underset{\lambda\in\R}{\sup}\{2\lambda\phi(f)-\lambda^2e\}\in\phi(E).
		\]
		By Theorem~\ref{T: main}, $(F,e)$ is an Archimedean $\Phi$-algebra under the multiplication defined in Definition~\ref{D: multn} for which $y^2=y^{[2]}\ (y\in F)$. Then we can define a map
		\[
		\phi(E)\times\phi(E)\ni\bigl(\phi(f),\phi(g)\bigr)\mapsto \phi(f)\phi(g)\in F.
		\]
		It also follows from Theorem~\ref{T: main} that the multiplication on $F$ satisfies
		\begin{align*}
			\phi(f)\phi(g)&=\dfrac{1}{2}\left(\bigl(\phi(f)+\phi(g)\bigr)^{[2]}-\left(\bigl(\phi(f)\bigr)^{[2]}+\bigl(\phi(g)\bigr)^{[2]}\right)\right)\\
			&=\dfrac{1}{2}\left(\underset{\lambda\in\R}{\sup}\{2\lambda\phi(f+g)-\lambda^2e\}-\underset{\lambda\in\R}{\sup}\{2\lambda\phi(f)-\lambda^2e\}-\underset{\lambda\in\R}{\sup}\{2\lambda\phi(g)-\lambda^2e\}\right)
		\end{align*}
		for every $f,g\in E$. However, since 
		\[
		\underset{\lambda\in\R}{\sup}\{2\lambda\phi(f+g)-\lambda^2e\},\ \underset{\lambda\in\R}{\sup}\{2\lambda\phi(f)-\lambda^2e\},\ \underset{\lambda\in\R}{\sup}\{2\lambda\phi(g)-\lambda^2e\}\in\phi(E)
		\]
		by assumption, we get that this map defines a multiplication on $\phi(E)$, making $\phi(E)$ an $f$-subalgebra of $F$. Consequently, $\phi(E)$ is semiprime.
		
		Next define $\psi\colon E\to\phi(E)$ by $\psi(f)=\phi(f)\quad (f\in E)$. Then, since $\phi$ in injective, $\psi$ is bijective. Clearly, the map
		\[
		E\times E\ni(f,g)\mapsto \psi^{-1}(fg)\in E
		\]
		defines a multiplication on $E$, $E$ is a semiprime $f$-algebra under this multiplication, and this multiplication on $E$ satisfies
		\[
		\phi(f^2)=\underset{\lambda\in\R}{\sup}\{2\lambda\phi(f)-\lambda^2e\}\in\phi(E)\quad (f\in E)
		\]
		and
		\[
		\phi(fg)=\dfrac{1}{2}\left(\underset{\lambda\in\R}{\sup}\bigl\{2\lambda\phi(f+g)-\lambda^2e\bigr\}-\underset{\alpha,\beta\in\R}{\sup}\bigl\{2\alpha \phi(f)+2\beta \phi(g)-(\alpha^2+\beta^2)e\bigr\}\right)\ (f,g\in E).
		\]
	\end{proof}

	\section{Square-Closedness as an Algebra Testing Tool}
	
	Perhaps the primary advantage that the theory of square closed (pseudo square closed) pointed vector lattices provides is a convenient tool for determining whether or not an Archimedean vector lattice is a $\Phi$-algebra (semiprime $f$-algebra). Indeed, instead of checking if an Archimedean vector lattice satisfies the several $\Phi$-algebra axioms, one can instead just check if it possesses the sole square closedness property with respect to a certain weak order unit. In this regard, we provide Proposition~\ref{P: functionally complete} as an illustration. Since this result involves the continuous Archimedean vector lattice functional calculus \cite{BusdPvR}, we remind the reader of how it is defined.
	
	\begin{definition}\cite[Definition~3.1]{BusdPvR}\label{D: fun cal}
		Fix $n\in\N$. Consider an Archimedean vector lattice $E$, and let $f_1,\dots,f_n\in E$. Suppose $h\colon\R^n\to\R$ is a continuous positively homogeneous function. If there exists $g\in E$ such that, for every real-valued vector lattice homomorphism  $\omega$ defined on the vector sublattice of $E$ generated by $\{f_1,\dots,f_n,g\}$, we have
		\[
		h\bigl(\omega(f_1),\dots,\omega(f_n)\bigr)=\omega(g),
		\]
		then we define $h(f_1,\dots,f_n):=g$. 
	\end{definition}
	
	We add that in Definition~\ref{D: fun cal}, if such a $g\in E$ exists, then it is unique \cite[Lemma~3.2]{BusdPvR}. Of special interest to us are Archimedean vector lattices where this functional calculus can always be defined. This concept leads to our next definition.
	
	\begin{definition}\cite[Definition~3.2]{BusSch}
		For $n\in\N$, an Archimedean vector lattice $E$, and a continuous positively homogeneous function $h\colon\R^n\to\R$, we say that $E$ is $h$-\textit{complete} if for every $f_1,...,f_n\in E$, there exists $g\in E$ such that $h(f_1,...,f_n)=g$. 	
	\end{definition}
	
	Additionally, we call an Archimedean vector lattice $E$ \textit{functionally complete} if, for all $n\in\N$ and every continuous positively homogeneous function $h\colon\R^n\to\R$, $E$ is $h$-complete. We note that an Archimedean vector lattice is functionally complete if and only if it is finitely uniformly complete, see \cite{LausTroit} for more details.
	
	\begin{proposition}\label{P: functionally complete}
		Let $E$ be an Archimedean vector lattice with strong order unit $e$. For each $n\in\N$, define $h_{n}\colon\R^2\to\R$ by
		\[
		h_{n}(x,y):=\begin{cases}
			\min\left\{\dfrac{x^2}{|y|}, n^2|y|\right\} & \text{if}\ y\neq 0\\
			0 & \text{if}\ y=0
		\end{cases}.
		\]
		The following are true.
		\begin{itemize}
			\item[(i)] For every $n\in\N$, $h_n$ is continuous and positively homogeneous.
			\item[(ii)] If for each $n\in\N$, $E$ is $h_n$-complete, then $E$ is a $\Phi$-algebra with $e$ as multiplicative unit.
		\end{itemize}
		Particularly, if $E$ is functionally complete, then $E$ is a $\Phi$-algebra with multiplicative unit $e$. 
	\end{proposition}
	
	\begin{proof}
		(i) It is evident that, for each $n\in\N$, $h_n$ is positively homogeneous and continuous at every point $(x,y)\in\R^2$ for which $y\neq 0$. Furthermore, the continuity of each $h_n$ at every point of the form $(x,0)\in\R^2$ follows directly from the squeeze theorem.
		
		(ii) Assume that $E$ is $h_n$-complete for every $n\in\N$. Let $f\in E$. Since $e$ is a strong order unit of $E$, there exists $n\in\N$ such that $|f|\leq ne$. By assumption, $h_n(f,e)$ is defined in $E$ via functional calculus.
		
		Consider next the uniform completion $E^u$ of $E$. Note that $E^u$ is an Archimedean $\Phi$-algebra with multiplicative unit $e$ (see e.g. \cite[Section~2]{VentvE}), and by Proposition~\ref{P: phi-alg implies sc}, $(E^u,e)$ is square closed. Let $C$ be the vector sublattice of $E^u$ generated by
		\[
		\left\{f,e, f^{[2]},h_n(f,e)\right\},
		\]
		and let $\omega\colon C\to\R$ be a vector lattice homomorphism for which $\omega(e)=1$. Since $|f|\leq ne$, we have $\omega(|f|)\leq n$ and thus $\bigl(\omega(f)\bigr)^2\leq n^2$. Therefore, using \cite[Theorem~3.11]{BusSch} in the first equality below, and Theorem~\ref{T: hom intergchange} in the last identity below, we have
		\[
		\omega\bigl(h_n(f,e)\bigr)=h_n\bigl(\omega(f), 1\bigr)=\min\Bigl\{\bigl(\omega(f)\bigr)^2, n^2\Bigr\}=\bigl(\omega(f)\bigr)^2=\bigl(\omega(f)\bigr)^{[2]}=\omega\left(f^{[2]}\right).
		\]
		Since the set of all vector lattice homomorphisms $\omega\colon C\to\R$ for which $\omega(e)=1$ separate the points of $C$, \cite[Remark 1.2(ii) and Theorem~2.4]{BusvR}, we have that
		\[
		f^{[2]}=h_n(f,e)
		\]
		holds in $E^u$. However, since $h_n(f,e)\in E$ by assumption, we have that $f^{[2]}\in E$. Hence $(E,e)$ is square closed. By Theorem~\ref{T: main}, $E$ is a $\Phi$-algebra with multiplicative unit $e$.
	\end{proof}
	
	As a concluding note, we add that every uniformly complete Archimedean vector lattice is functionally complete \cite[Theorem~3.7]{BusdPvR}. However, there exist functionally complete Archimedean vector lattices which are not uniformly complete. One example is the space of step functions on $[0,1]$. Therefore, Proposition~\ref{P: functionally complete} generalizes the well-known result that every uniformly complete Archimedean vector lattice with a strong order unit is a $\Phi$-algebra.

\end{document}